\documentclass[11pt,letterpaper]{article}
\usepackage[left=1.15in,right=1.15in,top=.85in,bottom=.85in]{geometry}
\usepackage{amsmath,amssymb,amsthm}
\usepackage{xcolor,titling}
\definecolor{linkblue}{RGB}{22,55,105}
\usepackage[numbers,sort]{natbib}
\usepackage[colorlinks,linkcolor=linkblue,citecolor=linkblue,urlcolor=linkblue,
 pdftitle={A counterexample to the Berger--Coburn conjecture},pdfauthor={Sam Looi}]{hyperref}
\newtheorem{theorem}{Theorem}[section]
\newtheorem{proposition}[theorem]{Proposition}
\newtheorem{lemma}[theorem]{Lemma}
\newtheorem{corollary}[theorem]{Corollary}
\newtheorem{conjecture}[theorem]{Conjecture}
\newcommand{\C}{\mathbb C}
\newcommand{\R}{\mathbb R}
\title{A counterexample to the Berger--Coburn conjecture}
\author{Sam Looi}
\date{28 November 2026}
\renewcommand{\maketitle}{%
 \begingroup\raggedright\setlength{\parskip}{0pt}
 {\LARGE\bfseries\thetitle\par}\vspace{7pt}
 {\large\theauthor\par}\vspace{2pt}
 \vspace{3pt}{\small\thedate\par}\endgroup}

\begin{document}
\maketitle
\begin{abstract}
We give the first counterexample to the 1994 conjecture of Berger and Coburn that a Toeplitz operator on Bargmann--Fock space is bounded if and only if the heat transform of its symbol at time $1/4$ is bounded. The same constructed function disproves a conjecture of Bauer, Coburn and Isralowitz. In every dimension, we construct a symbol whose Toeplitz operator is Hilbert--Schmidt and whose endpoint heat transform is nonnegative and unbounded. We also construct a symbol whose endpoint heat transform vanishes at infinity and whose Toeplitz operator is unbounded, disproving the converse implication in both conjectures. Both symbols are smooth, real, and square integrable against every translate of the Gaussian measure. A variant of the first example gives a compact operator whose endpoint heat transform is bounded and does not vanish at infinity.
\end{abstract}

\section{Introduction}
A basic problem on Bargmann--Fock space is to decide when a Toeplitz operator defined by a possibly unbounded symbol extends boundedly to the whole space. In their 1994 article~\cite{BC94}, Berger and Coburn conjectured that such an extension exists if and only if the heat transform of the symbol at time $t=\tfrac14$ is a bounded function. We give the first counterexample to the conjecture. In every complex dimension, we construct a symbol whose Toeplitz operator extends to a Hilbert--Schmidt operator and whose heat transform at the critical, or endpoint\footnote{In this nomenclature, we follow our earlier usage \cite{Looi2026BergerCoburnV1}.}, time $t = \tfrac14$ is unbounded. Since every Hilbert--Schmidt operator is compact, the same example also disproves the compactness conjecture of Bauer, Coburn, and Isralowitz~\cite[Conjecture~2]{BCI2010}.

Toeplitz operators on the Bargmann--Fock space provide a model in which holomorphic function theory, operator theory, and pseudodifferential calculus can be compared because these operators have interpretations as Weyl quantization operators. Under the Bargmann transform, a Toeplitz operator with a Schwartz symbol is unitarily equivalent to a Weyl operator on $L^2(\mathbb R^n)$. The Weyl symbol is obtained by applying heat flow to the Toeplitz symbol for time $t=\tfrac14$, the endpoint time in the conjecture~\cite{BC94,Fo89}.

Fix $n\ge 1$ and let $F^2(\mathbb C^n,d\mu)$ be the Bargmann--Fock, or Fock, space of entire functions that are square integrable with respect to the Gaussian measure
\[
d\mu(z)=(2\pi)^{-n}e^{-|z|^2/2}\,dv(z),
\]
where $dv$ is Lebesgue measure on $\mathbb C^n\simeq \mathbb R^{2n}$.
Let $P:L^2(d\mu)\to F^2(d\mu)$ be the orthogonal projection with reproducing kernel $K(z,w)=e^{z\cdot\overline w/2}$, and let $k_a=K(\cdot,a)/\sqrt{K(a,a)}$ be the normalized kernel at each $a \in \C^n$. The family $\{k_a\}$ of kernels satisfies $\|k_a\|_2=1$ and $k_0\equiv1$. Following Berger and Coburn \cite[\S1]{BC94}, we call a measurable symbol $g:\C^n\to\C$ \emph{admissible} if
\begin{equation}\label{eq:a_domain_condition}
 gk_a\in L^2(d\mu)\qquad\text{for every }a\in\C^n .
\end{equation}
For such a symbol the Toeplitz operator is given by
\[
T_g f=P(gf),
\qquad
(T_g f)(z)=\int_{\mathbb C^n} g(w)\,K(z,w)\,f(w)\,d\mu(w),
\]
with natural domain $\{f\in F^2(d\mu): gf\in L^2(d\mu)\}$. This domain contains the dense span of $\{k_a\}$. Throughout, we say that $T_g$ is bounded if it extends boundedly from this domain onto all of $F^2$.

For $t>0$, the heat transform of $g$ is
\begin{equation}\label{eq:intro-heat}
g^{(t)}(a)=(4\pi t)^{-n}\int_{\mathbb C^n} g(w)\,e^{-|w-a|^2/(4t)}\,dv(w),
\end{equation}
whenever the integral is absolutely convergent. It is the solution of the heat equation on $\C^n$ with initial datum $g$. In \cite[\S6, p.~582]{BC94}, Berger and Coburn conjectured that for an admissible symbol $g$, the operator $T_g$ is bounded if and only if $g^{(1/4)}$ is bounded.

\begin{theorem}\label{thm:main}
For every $n\ge1$ there are smooth real admissible symbols $g$, $h, g_0 :\C^n\to \R$ such that
\begin{enumerate}
 \renewcommand{\labelenumi}{\textup{(\roman{enumi})}}
 \item $T_g$ extends to a self-adjoint Hilbert--Schmidt operator, and $g^{(1/4)}$ is nonnegative and unbounded.
 \item $h^{(1/4)}\in C_0(\mathbb C^n)$ and
      $\sup_{a\in\mathbb C^n}\|T_h k_a\|=\infty$.
      In particular, $T_h$ has no bounded extension.
\item $T_{g_0}$ extends to a self-adjoint Hilbert--Schmidt operator,
      and $g_0^{(1/4)}$ is bounded and nonnegative but does not
      vanish at infinity.
\end{enumerate}
\end{theorem}

Here $C_0(\C^n)$ consists of continuous functions vanishing at infinity. Following our usage in version 1 \cite{Looi2026BergerCoburnV1}, we call $g^{(1/4)}$ the \emph{endpoint heat transform} of $g$, since $t=1/4$ is the common endpoint of the two time ranges in the Berger--Coburn estimates. Bauer, Coburn, and Isralowitz conjectured that $T_g$ is compact if and only if $g^{(1/4)}\in C_0(\C^n)$ \cite[\S6, Conjecture~2]{BCI2010}. Since every Hilbert--Schmidt operator is compact, Theorem~\ref{thm:main} disproves both directions. Further, in part (iii) we provide an example of a compact operator for which the endpoint heat transform is bounded and does not vanish at infinity.

Our January 2026 preprint~\cite{Looi2026BergerCoburnV1} gave the first counterexample to the Berger--Coburn conjecture by disproving the forward implication. We constructed a symbol $g$ whose Toeplitz operator $T_g$ extended to a Hilbert--Schmidt operator, and yet $g^{(1/4)}$ was unbounded on $\C^n$. 
We announced our disproof of the converse at IAS in February~\cite{LooiIAS2026Abstract}, and at the University of Chicago in May~\cite{LooiChicago2026Abstract}. Our July ICM abstract~\cite{LooiICM2026Abstract} also announced counterexamples in both directions for both boundedness and compactness conjectures.

\subsection*{Heat flow and the endpoint problem}
A connection between the Toeplitz operator and the heat transform is as follows. Taking the inner product of $T_gk_a$ and $k_a$ yields the Berezin transform, which is equivalently described as the heat transform at time 1/2:
\begin{equation}\label{eq:intro-berezin}
(T_g k_a,k_a)=g^{(1/2)}(a),
\end{equation}
so that if $T_g$ is bounded, then so is $g^{(1/2)}$, by the Cauchy--Schwarz inequality. As shown in \cite{BC94,IsralowitzZhu2010Toeplitz}, for nonnegative symbols $g$, this necessary condition is also sufficient. For general symbols, Berger and Coburn established two complementary norm estimates \cite[Theorems~11 and~12]{BC94}, which, in their normalisation, are:
\begin{equation}\label{eq:bc-estimates}
 \begin{aligned}
 \|g^{(t)}\|_\infty &\le C_{n,t}\|T_g\|, &&\tfrac14<t<1,\\
 \|T_g\| &\le C_{n,t}\|g^{(t)}\|_\infty, &&0<t<\tfrac14 .
 \end{aligned}
\end{equation}
The time $t=\tfrac14$ is the value selected by the Weyl calculus, in the sense that for a Schwartz symbol $g$, the Weyl symbol corresponding to $T_g$ under the Bargmann transform is precisely $g^{(1/4)}$. The same pair of authors remarked that no constant $M$ independent of $g$ can satisfy $\|g^{(1/4)}\|_\infty\le M\|T_g\|$ for every admissible symbol \cite[pp.~579, 582]{BC94}. This left open whether boundedness of each individual $T_g$ nevertheless forces $g^{(1/4)}$ to be bounded, and they stated their conjecture
\begin{conjecture}[Berger--Coburn {\cite[\S 6]{BC94}}]\label{conj:BC}
Assume that $g$ is measurable and that $gk_a\in L^2(d\mu)$ for every $a\in\mathbb C^n$.
Then $T_g$ extends to a bounded operator on $F^2(\mathbb C^n,d\mu)$ if and only if $g^{(1/4)}$ is bounded on $\mathbb C^n$.\footnote{We use ``extends to'' to account for the dense domain of $T_g$ noted in \cite[\S 1]{BC94}.}
\end{conjecture}
Part (i) of our theorem answers this question.

For background on the topic of boundedness and compactness of Toeplitz operators on the Bargmann--Fock, or Fock, space, see Zhu~\cite{ZhuFock}, the early papers of Berger and Coburn~\cite{BC86,BC87}, and the treatment of positive measure symbols by Isralowitz and Zhu~\cite{IsralowitzZhu2010Toeplitz}. Bauer and Fulsche~\cite{BauerFulsche2019} proved the Berger--Coburn estimates on both sides of the endpoint on the more general $p$-Fock spaces.  Dewage and Mitkovski~\cite{DewageMitkovski2025} gave a proof of the Berger--Coburn estimates on either side of the endpoint using quantum harmonic analysis. Wu and Zheng~\cite{WuZheng2021} obtained sufficient conditions for boundedness by decomposing the symbol in frequency.

The endpoint criterion, i.e. the Berger--Coburn conjecture, does hold for several classes of symbols. Berger and Coburn's 1994 paper showed that it holds for nonnegative symbols. In addition, it holds for  $\mathrm{BMO}^1$ and metaplectic symbols. For symbols in the heat $\mathrm{BMO}^1$ class, Bauer, Coburn, and Isralowitz characterized boundedness and compactness through the corresponding heat transforms~\cite[Theorems~5 and~6]{BCI2010}. See also Coburn, Isralowitz, and Li~\cite{CoburnIsralowitzLi2011} for BMO symbols. For metaplectic Toeplitz operators with symbols $e^q$, the endpoint boundedness criterion holds under the hypotheses of the respective papers for complex quadratic forms $q$~\cite{CHS2019,CHS2023} and for inhomogeneous quadratic polynomials $q$~\cite{CHSW2021,Xiong2023}. Other approaches exploit radial symmetry~\cite{GrudskyVasilevski2002RadialComponentEffects}, translation invariance of Fock--Carleson measures~\cite{EsmeralRozenblumVasilevski2019LinvariantFockCarleson}, or the structure of operator algebras~\cite{Bauer2005PsiFrechetThesis,BauerFulscheRodriguezRodriguez2024OperatorsFockToeplitzAlgebra}. Within the Fock Toeplitz algebra, compactness is characterized by vanishing of the Berezin transform at infinity~\cite{BauerIsralowitz2012}. Xia and Zheng~\cite{XiaZheng2013} obtained a related criterion using localization.
For the trace class bound, a counterexample to the forward implication with $T_g\in\bigcap_{1<p<\infty}\mathcal S_p$, and a counterexample to the converse, see~\cite[Theorem~1.1]{Virtanen2026}.

\subsection*{Weyl quantisation}
For an admissible Toeplitz symbol $g$, we write $a:=g^{(1/4)}$, viewed as a function on $\mathbb R^{2n}$. Under the Bargmann transform $B$, the operator $T_g$ is unitarily equivalent to $A_g:=B^{-1}T_gB$ on the domain $B^{-1}\mathcal D(T_g)$, the inverse image of $T_g$'s natural domain. In these coordinates, Berger and Coburn's conjecture asks whether $A_g$ extends to a bounded operator on $L^2(\mathbb R^n)$ precisely when $a$ is bounded~\cite[\S~6]{BC94}. The compactness conjecture of Bauer, Coburn, and Isralowitz asks whether $A_g$ extends to a compact operator precisely when $a$ is continuous and vanishes at infinity~\cite[\S~6, Conjecture~2]{BCI2010}. For Schwartz $g$, the Toeplitz--Weyl correspondence identifies $A_g$ with $a^w$, the Weyl quantisation of $a$~\cite[\S~5]{BC94}.

\subsection*{Idea of the constructions}
The first counterexample is assembled from a family of functions we call ``blocks'' whose Toeplitz norms tend to zero while their $t=\tfrac14$ heat profiles each have height one, and hence stay uniformly bounded below by a positive constant. We begin by describing the general idea of the proof. The Weyl quantization operator has a Hilbert--Schmidt norm that is easily computed, and we use the relation between it and the Toeplitz operator to then obtain an upper bound on the Hilbert--Schmidt norm of the Toeplitz operator. The Hilbert--Schmidt norm majorizes the operator norm and this also provides an upper bound on the operator norm. For this reason, we find the Weyl quantization viewpoint profitable. \, Slightly more precisely, we use the Weyl characterization of the Toeplitz operator to see that the Hilbert--Schmidt norm of the operator is proportional to the $L^2$ norm of the profile, so we can make the operator small by concentrating the profile. In more detail, for $a\in L^2(\mathbb R^{2n})$ we write $a(D,X)$ for the Weyl operator on $L^2(\mathbb R^n)$. We prove
\[
\|a(D,X)\|_{\rm HS}=(2\pi)^{-n/2}\|a\|_{L^2(\mathbb R^{2n})}
\]
and choose oscillatory symbols $a_R$ with $\|a_R\|_\infty=1$ but $\|a_R\|_2\simeq R^{-n}$, so that $a_R(D,X)$ has operator norm $O(R^{-n})$. We use the Bargmann transform to see that these Weyl operators correspond to bounded Toeplitz operators with symbols $g_R$ satisfying
\[
g_R^{(1/4)}=a_R
\qquad\text{and}\qquad
\|T_{g_R}\|\lesssim R^{-n}.
\]
We then translate and rescale the blocks so that their Hilbert--Schmidt norms are summable while their heat profiles develop peaks of increasing height. The profiles are nonnegative, so these peaks survive in the sum.

For the disproof of the converse implication, we construct blocks whose endpoint heat profiles have supremum norm one but whose Toeplitz operators have arbitrarily large action on $k_0$. We translate small multiples of these blocks to centers $a_m$ so that their endpoint heat profiles sum uniformly to a function vanishing at infinity. Let $h$ be the resulting symbol. To see that $T_h$ is unbounded, we shall test it on $k_{a_m}$, which are unit norm, reproducing kernels at the center of the $m$th block. These centers are chosen so that $\|T_h k_{a_m}\|\to\infty$.

\section{A Hilbert--Schmidt counterexample}\label{sec:first}
As we have described, the first counterexample's symbol will be a sum of functions whose Hilbert--Schmidt norms tend to zero while their $t=\tfrac14$ heat profiles stay uniformly large. We choose oscillatory symbols $a_R$ with $\|a_R\|_\infty=1$ and $\|a_R\|_2\sim R^{-n}$, where $\sim$ denotes comparability up to positive factors depending only on $n$ and the fixed profile. The Hilbert--Schmidt estimate below makes both the Hilbert--Schmidt norms and operator norms arbitrarily small as $R$ tends to infinity. We then translate and rescale these blocks so that the norms of the Toeplitz operator remain summable while the endpoint heat transforms develop peaks of increasing height at widely separated points.

\subsection*{A Hilbert--Schmidt estimate for Weyl quantization}
We use the Fourier transform $\widehat f(\eta)=\int f(Z)e^{-iZ\cdot\eta}\,dZ$ on $\R^{2n}$ and identify $Z=(x,\xi)$ with $z=x-i\xi$. For $a\in\mathcal S(\R^{2n})$ the Weyl quantization operator $a(D,X)$ on Schwartz functions is given by
\begin{equation}\label{eq:weyl}
(a(D,X)f)(x)=(2\pi)^{-n}\int_{\mathbb R^n}\int_{\mathbb R^n}
a \Big(\frac{x+y}{2},\xi\Big)\,e^{i(x-y)\cdot \xi}\,f(y)\,dy\,d\xi.
\end{equation}

\begin{lemma}\label{lem:HS}
If $a\in L^2(\mathbb R^{2n})$, then $a(D,X)$ is Hilbert--Schmidt and
\[
\|a(D,X)\|_{\mathrm{HS}}=(2\pi)^{-n/2}\,\|a\|_{L^2(\mathbb R^{2n})}.
\]
In particular, $\|a(D,X)\|\le (2\pi)^{-n/2}\|a\|_2$.
\end{lemma}

\begin{proof}
For $a$ Schwartz, the integral kernel of $a(D,X)$ is
\[
K_a(x,y)=(2\pi)^{-n}\int_{\mathbb R^n} a \Big(\frac{x+y}{2},\xi\Big)e^{i(x-y)\cdot \xi}\,d\xi.
\]
If $u=(x+y)/2$ and $v=x-y$, then $dx\,dy=du\,dv$ and
\[
K_a \Big(u+\frac v2,u-\frac v2\Big)=(2\pi)^{-n}\int_{\mathbb R^n} a(u,\xi)e^{iv\cdot \xi}\,d\xi.
\]
For almost every $u$, Plancherel's theorem in $\xi$ gives
\[
\int_{\mathbb R^n}\Big|\int_{\mathbb R^n} a(u,\xi)e^{iv\cdot \xi}\,d\xi\Big|^2\,dv
=(2\pi)^n\int_{\mathbb R^n}|a(u,\xi)|^2\,d\xi.
\]
Integrating in $u$ gives $\|a(D,X)\|_{\mathrm{HS}}^2=(2\pi)^{-n}\|a\|_2^2$. By density, Weyl quantization extends with this identity to $L^2(\R^{2n})$.
\end{proof}

Let $B:L^2(\mathbb R^n)\to F^2(d\mu)$ be the Bargmann unitary mapping, so that $B^{-1}k_0$ is the normalized Gaussian $\pi^{-n/4}e^{-|y|^2/2}$. With our identification $z=x-i\xi$, the Toeplitz--Weyl correspondence~\cite[\S5, pp.~580--581]{BC94} gives
\begin{equation}\label{eq:bargmann-weyl}
 B^{-1}T_hB=h^{(1/4)}(D,X), \qquad h\in\mathcal S(\mathbb C^n).
\end{equation}

\subsection*{Blocks with small norm}
Fix $\Phi\in C_c^\infty(\mathbb R^{2n})$ with $\Phi\ge0$, $\int_{\mathbb R^{2n}}\Phi(u)\,du=1$, and $\operatorname{supp}\Phi\subset\{|u|\le1/2\}$. For $R\ge1$ define
\begin{equation}\label{eq:aR}
a_R(z)=\int_{\mathbb R^{2n}}\Phi(u)e^{iRu\cdot z}\,du,
\qquad z\in\mathbb R^{2n}.
\end{equation}
Then $a_R(0)=1$ and $|a_R(z)|\le1$, so $\|a_R\|_\infty=1$.

We choose $\Phi$ so that the translated profiles cannot cancel when we sum them. In order to achieve this, we take $\Phi=\psi*\widetilde\psi$, where $\psi\in C_c^\infty(\mathbb R^{2n})$ is nonnegative, supported in $B(0,1/4)$, and has integral one, and $\widetilde\psi(u)=\psi(-u)$. This $\Phi$ satisfies the conditions above, and
\[
a_R(z)=\left|\int_{\mathbb R^{2n}}\psi(u)e^{iRu\cdot z}\,du\right|^2\ge0.
\]

Since $\Phi$ is smooth and compactly supported, $a_R$ is Schwartz and has compact Fourier support. The scaling identity $a_R(z)=a_1(Rz)$ gives
\begin{equation}\label{eq:aR-L2}
\|a_R\|_{L^2(\mathbb R^{2n})}=R^{-n}\|a_1\|_{L^2(\mathbb R^{2n})}.
\end{equation}
We define
\begin{equation}\label{eq:gR}
g_R=e^{-\Delta/4}a_R,
\end{equation}
where $\Delta$ is the Laplacian on $\mathbb R^{2n}$. Indeed, $\widehat a_R(\xi)=(2\pi)^{2n}R^{-2n}\Phi(\xi/R)$ is smooth and compactly supported, so the same is true of $\widehat g_R(\xi)=e^{|\xi|^2/4}\widehat a_R(\xi)$. Thus $g_R$ is Schwartz. Moreover, $g_R$ is real: \, since $\widehat g_R$ is real and even, the sine term in the Fourier inversion formula integrates to zero. Applying the heat multiplier gives
\begin{equation}\label{eq:gR-quarter}
g_R^{(1/4)}=a_R.
\end{equation}

\begin{proposition}\label{prop:block}
There exists $C_0=C_0(n,\Phi)>0$ such that for every $R\ge1$,
\[
\|T_{g_R}\|\le\|T_{g_R}\|_{\mathrm{HS}}=C_0R^{-n},
\qquad
\|g_R^{(1/4)}\|_\infty=1.
\]
\end{proposition}
\begin{proof}
Let $A_R=a_R(D,X)$ be the Weyl quantization operator on $L^2(\mathbb R^n)$ with symbol $a_R$. By Lemma~\ref{lem:HS} and \eqref{eq:aR-L2},
\[
\|A_R\|_{\mathrm{HS}}
=(2\pi)^{-n/2}\|a_R\|_2
=(2\pi)^{-n/2}R^{-n}\|a_1\|_2.
\]
Since $g_R^{(1/4)}=a_R$, the correspondence \eqref{eq:bargmann-weyl} gives $T_{g_R}=BA_RB^{-1}$. 

We prove the stated estimate with $C_0=(2\pi)^{-n/2}\|a_1\|_2$. If $(e_j)$ is an orthonormal basis of $L^2(\mathbb R^n)$, then $(Be_j)$ is an orthonormal basis of Fock space. Since $B$ preserves norms,
\[
\|T_{g_R}\|_{\mathrm{HS}}^2
=\sum_j \|T_{g_R}Be_j\|^2
=\sum_j \|BA_Re_j\|^2
=\sum_j \|A_Re_j\|^2
=\|A_R\|_{\mathrm{HS}}^2.
\]
Hence the preceding formula for $\|A_R\|_{\mathrm{HS}}$ gives $\|T_{g_R}\|_{\mathrm{HS}}=C_0R^{-n}$ with $C_0=(2\pi)^{-n/2}\|a_1\|_2$.

Finally, $\|g_R^{(1/4)}\|_\infty=1$ follows because by construction, $\|a_R\|_\infty=1$.
\end{proof}

\subsection*{A summation lemma}
For a measurable function $f$ and $a\in\C^n$, we define $\|f\|_{2,a}=\|fk_a\|_{L^2(d\mu)}$. By computation,
\begin{equation}\label{eq:ka-weight}
 |k_a(w)|^2\,d\mu(w)=(2\pi)^{-n}e^{-|w-a|^2/2}\,dv(w)
\end{equation}
and hence
\[
\|f\|_{2,a}^2
=\|fk_a\|_{L^2(d\mu)}^2
=(2\pi)^{-n}\int_{\mathbb C^n}|f(w)|^2e^{-|w-a|^2/2}\,dv(w).
\]

This local norm provides two kinds of control: whenever $\|f\|_{2,a}<\infty$,
\begin{equation}\label{eq:local-bounds}
 \|T_fk_a\|\le \|f\|_{2,a},\qquad |f^{(1/4)}(a)|\le 2^n \|f\|_{2,a}
\end{equation}
The first inequality follows from $T_fk_a=P(fk_a)$ and the fact that the orthogonal projection is bounded on $L^2$ with constant 1. We prove the second inequality by using $e^{-|w-a|^2}\le e^{-|w-a|^2/2}$ and then applying Cauchy--Schwarz to the Gaussian probability measure $|k_a|^2d\mu$ in \eqref{eq:ka-weight}. This also gives absolute convergence of the heat integral. For an admissible $f$, the function $f^{(1/4)}$ is continuous. Indeed, for $|a|\le A$ the heat kernel is bounded by $\pi^{-n} e^{A^2}e^{-|w|^2/2}$, and $f\in L^2(d\mu)$ since $k_0=1$.

We perform a placement-of-centers argument such that the blocks' sum is admissible, but the same argument will also apply to the disproof of the converse implication, where the action on a kernel vector must be large.

\begin{lemma}\label{lem:sum}
Let $(F_m)_{m\ge1}$ be Schwartz functions on $\C^n$ and let $(\varepsilon_m)_{m\ge1}$ be positive and summable. There are centers $a_m$ with $|a_m|\ge m$ and $|a_m-a_j|\ge1$ for $j\ne m$ such that the translates $g_m=F_m(\cdot-a_m)$ satisfy
\begin{equation}\label{eq:star}
 \sum_m \|g_m\|_{2,a}<\infty \quad (a\in\C^n),\qquad
 \sum_{j\ne m}\|g_j\|_{2,a_m}\le\sum_{j\ge m}\varepsilon_j .
\end{equation}
The series $g=\sum_m g_m$ converges locally uniformly with all derivatives. Its sum is admissible and for every $a\in\C^n$,
\begin{equation}\label{eq:sum-heat}
 g^{(1/4)}(a)=\sum_m g_m^{(1/4)}(a),\qquad
 T_gk_a=\sum_m T_{g_m}k_a .
\end{equation}
If $\sum_m\|T_{g_m}\|<\infty$, then $T_g$ extends to the operator sum $\sum_m T_{g_m}$ on $F^2(d\mu)$.
\end{lemma}

\begin{proof}
We suppose $a_1,\ldots,a_{m-1}$ have been chosen and define the set
\[
E_m=\overline{B(0,m)}\cup\{a_1,\ldots,a_{m-1}\}.
\]
We choose the new translate to be small on $E_m$. We choose this set because the growing balls $\overline{B(0,m)}$ give control at every fixed point, and including the earlier centers gives control of the later translates at those centers. We also require the earlier translates to be small at the new center.

For a candidate center $b$, \eqref{eq:ka-weight} and the change of variables $u=w-b$ give
\[
\sup_{a\in E_m}\|F_m(\cdot-b)\|_{2,a}^2
\le (2\pi)^{-n}\int_{\mathbb C^n}|F_m(u)|^2
       \sup_{a\in E_m}e^{-|u+b-a|^2/2}\,dv(u).
\]
For each fixed $u$, 
\[ \lim_{|b|\to\infty} \sup_{a\in E_m}e^{-|u+b-a|^2/2} =0\]
because $E_m$ is bounded. By the dominated convergence theorem, the right-hand side tends to 0. The same argument with the Gaussian centered at $b$ gives $\|g_j\|_{2,b}\to0$ for each $j<m$.  Moreover, since $F_m$ is Schwartz, the derivatives of $F_m(\cdot-b)$ of order at most $m$ tend uniformly to zero on $\overline{B(0,m)}$.

Hence, we can choose $a_m=b$ with $|a_m|\ge m$ and $|a_m-a_j|\ge1$ for every $j<m$ so that, with $g_m=F_m(\cdot-a_m)$,
\[
\sup_{a\in E_m}\|g_m\|_{2,a}+\sum_{j<m}\|g_j\|_{2,a_m}
\le\varepsilon_m,
\qquad
\max_{|\alpha|\le m}\sup_{|z|\le m}|\partial^\alpha g_m(z)|
\le 2^{-m}.
\]

For fixed $a$, we have $a\in E_m$ for all sufficiently large $m$, so the tail of $\sum_m\|g_m\|_{2,a}$ is bounded by the corresponding tail of $\sum_m\varepsilon_m$. At $a_m$, the choice at stage $m$ bounds the entire sum over $j<m$ by $\varepsilon_m$. For $j>m$, the choice at stage $j$ gives $\|g_j\|_{2,a_m}\le\varepsilon_j$ because $a_m\in E_j$. Adding these bounds gives the second estimate in \eqref{eq:star}. On each fixed compact set, the derivative bound applies to each fixed derivative for all sufficiently large $m$. Thus $g=\sum_m g_m$ converges locally uniformly with all derivatives.

Let us write the partial sums as $S_N=\sum_{m=1}^N g_m$. Let $a\in \C^n$ be an arbitrary point. By \eqref{eq:star}, the series $\sum_m g_mk_a$ converges absolutely in $L^2(d\mu)$. On each compact set, $S_Nk_a$ also converges uniformly to $gk_a$, so the $L^2(d\mu)$ limit agrees with $gk_a$ there. Hence $g$ is admissible and $\|S_N-g\|_{2,a}\to0$. Applying $P$ to this weighted convergence gives the identity for $T_gk_a$ in \eqref{eq:sum-heat}. Applying \eqref{eq:local-bounds} to $S_N-g$ gives $ g^{(1/4)}(a)=\sum_m g_m^{(1/4)}(a)$. 

Finally, suppose $\sum_m\|T_{g_m}\|<\infty$ and let $X=\sum_mT_{g_m}$, with convergence in operator norm. By \eqref{eq:sum-heat}, $Xk_a=T_gk_a$ for every $a$. For fixed $b$, admissibility gives $\overline g\,k_b\in L^2(d\mu)$, and
\[
(k_a,X^*k_b)=(Xk_a,k_b)=(gk_a,k_b)
=(k_a,P(\overline g\,k_b)).
\]
Since the kernel vectors have dense span, $X^*k_b=P(\overline g\,k_b)$. Now if $f\in F^2(d\mu)$ and $gf\in L^2(d\mu)$, then
\[
(Xf,k_b)=(f,X^*k_b)=(f,\overline g\,k_b)
=(gf,k_b)=(T_gf,k_b).
\]
This holds for every $b$, so $Xf=T_gf$ on the natural domain of $T_g$.
\end{proof}

\subsection*{The first counterexample}
We shall choose coefficients $c_m=m$ making the endpoint heat peaks grow and frequencies $R_m=m^3$ making the Toeplitz norms summable even after multiplying the symbols by $c_m$. We shall use Lemma~\ref{lem:sum} to choose the centers.

\begin{proof}[Proof of Theorem~\ref{thm:main}\,\textup{(i)}]
Let $(g_R)$ be the family from Proposition~\ref{prop:block}. For integers $m\ge1$ we put $F_m=c_mg_{R_m}$, and apply Lemma~\ref{lem:sum} with $\varepsilon_m=2^{-m}$. We put
\[
 g_m(z)=c_mg_{R_m}(z-a_m),\qquad g=\sum_{m=1}^\infty g_m.
\]
Then $g$ is smooth, real and admissible. For $b\in\C^n$, the translation $(U_bf)(z)=k_b(z)f(z-b)$ is unitary on $L^2(d\mu)$ and preserves $F^2(d\mu)$ in both directions. Thus it commutes with $P$, and
\begin{equation}\label{eq:translate}
 T_{F(\cdot-b)}=U_bT_FU_b^*,\qquad U_bk_0=k_b
\end{equation}
for bounded $F$ \cite[\S4]{BC94}. By linearity in the symbol and unitary invariance of the Hilbert--Schmidt norm,
\[
 \|T_{g_m}\|_{\mathrm{HS}}=|c_m|\,\|T_{g_{R_m}}\|_{\mathrm{HS}}=C_0m^{1-3n}.
\]
For $n\ge1$ this is clearly summable. Hence the operator series converges in Hilbert--Schmidt norm, and Lemma~\ref{lem:sum} identifies its sum with the extension of $T_g$. Since the symbols $g_m$ are real, the partial sums are self-adjoint. Their operator norm limit $T_g$ is therefore self-adjoint as well.

Since the endpoint heat transform commutes with translation, so $g_m^{(1/4)}(z)=c_m a_{R_m}(z-a_m)$ and thus $g_m^{(1/4)}(a_m)=m$. All summands in \eqref{eq:sum-heat} are nonnegative, hence $g^{(1/4)}(a_m)\ge m$. By continuity of $g^{(1/4)}$, the values of $g^{(1/4)}$ exceed $m/2$ on a nonempty open ball about $a_m$. Each such ball has positive measure, so $g^{(1/4)}\notin L^\infty(\C^n)$.
\end{proof}

\section{A counterexample to the converse implication}\label{sec:converse}
The construction of the counterexample has two steps: first, we build Schwartz symbols whose $t=\tfrac14$ heat transforms have norm one, but whose Toeplitz action on $k_0$ is arbitrarily large; second, we translate small multiples of these symbols far apart.

\begin{lemma}\label{lem:large-action-block}
There exist a constant $c>0$ and a sequence $(h_N)_{N\ge 1}\subset \mathcal S(\mathbb R^{2n})$ such that, for every $N\ge 1$,
\[
\|h_N^{(1/4)}\|_\infty=1
\qquad\text{and}\qquad
\|T_{h_N}k_0\| \ge c\sqrt N.
\]
\end{lemma}

\begin{proof}
We use the notation $v_0=B^{-1}k_0$, the normalized Gaussian from Section~\ref{sec:first}. First, we construct a Weyl symbol $b_N$ for which $\|b_N(D,X)v_0\|=1$ and $\|b_N\|_\infty\le CN^{-1/2}$. A Fourier cutoff will then allow us to undo the heat flow.

For $u,v\in\mathcal S(\mathbb R^n)$, the cross Wigner transform is
\[
W(u,v)(x,\xi)
:= \int_{\mathbb R^n}
u\Big(x+\frac q2\Big)
\overline{v\Big(x-\frac q2\Big)}
e^{-iq\cdot\xi}\,dq.
\]
Fourier inversion in \eqref{eq:weyl} gives the integral kernel $u(x)\overline{v(y)}$ for $W(u,v)(D,X)$. In particular $W(u,v_0)(D,X)v_0=u$. We fix $\phi\in C_c^\infty(\R^n)$ with $\|\phi\|_2=1$, and choose $L>0$ so large that the translates $\phi_j=\phi(\cdot-jLe_1)$ have pairwise disjoint supports. The functions
\[
 u_N=N^{-1/2}\sum_{j=1}^N\phi_j,\qquad b_N=W(u_N,v_0)
\]
satisfy the following: $\|u_N\|_2=1$ and $b_N(D,X)v_0=u_N$.

Our goal in this paragraph will be to obtain a decaying upper bound estimate for $\|b_N\|_\infty$, with respect to the variable $N$. By the identity
\[
W(\phi(\cdot-y),v_0)(x,\xi)
= e^{-iy\cdot\xi}W(\phi,v_0)\Big(x-\frac y2,\xi\Big)
\]
an easy application of the triangle inequality gives
\[
|b_N(x,\xi)|
\le \frac1{\sqrt N}\sum_{j=1}^N
\left|W(\phi,v_0)\Big(x-\frac{jL}{2}e_1,\xi\Big)\right|.
\]
The Wigner function $W(\phi,v_0)$ is Schwartz, so $|W(\phi,v_0)(x,\xi)|\le C(1+|x|)^{-2}$ uniformly in $\xi$. By projecting onto the first coordinate axis, we can bound the sum by $C\sum_{j\in\mathbb Z}(1+|x_1-jL/2|)^{-2}$. 
\[
\begin{aligned}
\sup_{x_1\in\mathbb R}\sum_{j\in\mathbb Z}(1+|x_1-jL/2|)^{-2}
&=\sup_{|x_1|\le L/4}\sum_{j\in\mathbb Z}(1+|x_1-jL/2|)^{-2}\\
&\le 1+2\sum_{j\ge1}(1+jL/4)^{-2}<\infty .
\end{aligned}
\]
Thus $\|b_N\|_\infty\le CN^{-1/2}$.

Fix $\chi\in C_c^\infty(\R^{2n})$ with $\chi=1$ near the origin, and let $\Pi_R f := (\chi(\cdot/R)\widehat f)^\vee$. The operator $\Pi_R$ is convolution with an $L^1$ kernel whose $L^1$ norm is independent of $R$, so $\|\Pi_R f\|_\infty\le\|\check\chi\|_1\|f\|_\infty$. Since $\Pi_R b_N\to b_N$ in $L^2(\R^{2n})$ norm, we can choose $R_N$ so that $\widetilde b_N=\Pi_{R_N}b_N$ satisfies $\|b_N-\widetilde b_N\|_2\le (2\pi)^{n/2}/2$. By Lemma~\ref{lem:HS},
\[
 \|b_N(D,X)-\widetilde b_N(D,X)\|
 \le (2\pi)^{-n/2}\|b_N-\widetilde b_N\|_2\le \tfrac12 .
\]
It follows that $\|\widetilde b_N(D,X)v_0\|\ge1/2$, and the cutoff preserves the bound $\|\widetilde b_N\|_\infty\le C N^{-1/2}$. Defining $\sigma_N=\widetilde b_N/\|\widetilde b_N\|_\infty$, we obtain
\[
 \|\sigma_N\|_\infty=1,\qquad \|\sigma_N(D,X)v_0\|\ge c\sqrt N .
\]
Since $\widehat{\sigma_N}$ is smooth and compactly supported, $h_N=e^{-\Delta/4}\sigma_N$ is Schwartz and $h_N^{(1/4)}=\sigma_N$. By \eqref{eq:bargmann-weyl}, $T_{h_N}=B\sigma_N(D,X)B^{-1}$. Hence $\|T_{h_N}k_0\|=\|\sigma_N(D,X)v_0\|\ge c\sqrt N$.
\end{proof}

\begin{proof}[Proof of Theorem~\ref{thm:main}\,\textup{(ii)}]
For each $m$, we take the real or imaginary part of $h_{m^6}$, choosing the part $u$ for which $\|T_uk_0\|$ is larger, and then we put $F_m=m^{-2}u$. This real Schwartz symbol satisfies
\[
 \|F_m^{(1/4)}\|_\infty\le m^{-2},\qquad
 \|T_{F_m}k_0\|\ge (c/2)m .
\]
Indeed heat flow commutes with taking real and imaginary parts, and (by the triangle inequality) at least one of the two parts gives an action of norm at least $\|T_{h_{m^6}}k_0\|/2$.

Applying Lemma~\ref{lem:sum} with $\varepsilon_m=2^{-m}$ and writing $g_m=F_m(\cdot-a_m)$ and $h=\sum_m g_m$, then $h$ is smooth, real and admissible. Each $g_m^{(1/4)}$ is Schwartz and has supremum norm at most $m^{-2}$, so the series for $h^{(1/4)}$ in \eqref{eq:sum-heat} converges uniformly. In particular $h^{(1/4)}\in C_0(\C^n)$.

It remains to test $T_h$ on the normalized kernels $k_{a_m}$. By \eqref{eq:translate}, $\|T_{g_m}k_{a_m}\|=\|T_{F_m}k_0\|$. For the other terms, \eqref{eq:local-bounds} and \eqref{eq:star} give
\[
 \sum_{j\ne m}\|T_{g_j}k_{a_m}\|
 \le \sum_{j\ne m}\|g_j\|_{2,a_m}\le 2^{1-m}.
\]
Using \eqref{eq:sum-heat} at $a_m$ we obtain $\|T_hk_{a_m}\|\ge (c/2)m-2^{1-m}\to\infty$. Since $\|k_{a_m}\|=1$, $T_h$ has no bounded extension.
\end{proof}

\subsection*{A compact operator with bounded endpoint heat transform}
The first example already disproves necessity in the compactness conjecture because its endpoint heat transform is unbounded. Taking all its coefficients equal to one gives the following stronger conclusion.

\begin{corollary}\label{cor:compact}
For every $n\ge1$, there is a smooth real admissible symbol $g_0$ for which $T_{g_0}$ extends to a self-adjoint Hilbert--Schmidt operator and $g_0^{(1/4)}$ is bounded and nonnegative but does not vanish at infinity.
\end{corollary}

\begin{proof}[Proof of Theorem~\ref{thm:main}(iii)]
We use the construction from part (i) with $F_m=g_{R_m}$ and $R_m=m^3$.
Then, the Hilbert--Schmidt norms are $C_0m^{-3n}$, which is summable and hence their sum is finite. With the centers from Lemma~\ref{lem:sum}, the same argument gives a smooth real admissible symbol $g_0$ and a self-adjoint Hilbert--Schmidt extension of $T_{g_0}$. Its endpoint heat transform is
\[
 g_0^{(1/4)}(z)=\sum_m a_1\big(R_m(z-a_m)\big).
\]
The summands are nonnegative and the $m$th summand equals one at $a_m$. Since $|a_m|\to\infty$, it remains only to prove boundedness. At most one center lies within distance $1/3$ of a given $z$, and its summand is at most one. For every other center, the estimate $|a_1(u)|\le C(1+|u|)^{-1}$ (which holds since $a_1$ is Schwartz) bounds the summand by $3C/R_m$. Since $\sum_m R_m^{-1}<\infty$, the bound is uniform in $z$.
\end{proof}

\section*{Acknowledgements}
The author thanks Otte Heinävaara for helpful discussions. The author also thanks Michael Hitrik
for introducing the problem, Jeck Lim for comments on an earlier version of the manuscript, and
Jared Wunsch for encouragement and advice on presentation. The author learned of this problem
at the American Institute of Mathematics, and was
supported by a Taussky--Todd Fellowship.

\bibliographystyle{abbrvnat}
\bibliography{berger-coburn}

\par\medskip
\noindent{\small
California Institute of Technology\\
\href{mailto:s.looi9309@gmail.com}{s.looi9309@gmail.com}\par}
\end{document}